\input ./style/arxiv-vmsta.cfg
\documentclass[numbers,compress,v1.0.1]{vmsta}
\usepackage{vtexbibtags}
\usepackage{authorquery}

\volume{5}
\issue{4}
\pubyear{2018}
\firstpage{501}
\lastpage{508}
\aid{VMSTA120}
\doi{10.15559/18-VMSTA120}
\articletype{research-article}

\startlocaldefs
\allowdisplaybreaks

\newtheorem{theorem}{Theorem}

\newtheorem{lemma}{Lemma}
\newtheorem{corollary}{Corollary}
\newtheorem{proposition}[theorem]{Proposition}

\theoremstyle{definition}
\newtheorem{definition}[theorem]{Definition}
\newtheorem*{question}{Question}

\newtheorem{example}[theorem]{Example}

\newcommand{\abs}[1]{\lvert #1\rvert}
\newcommand{\norm}[1]{\lVert #1\rVert}
\newcommand{\Bigabs}[1]{\Bigl\lvert #1\Bigr\rvert}

\renewcommand{\leq}{\leqslant}
\renewcommand{\geq}{\geqslant}
\newcommand{\term}[1]{{\textit{\textbf{#1}}}}
\endlocaldefs

\begin{aqf}
\querytext{Q1}{Please check if "for every two vectors $x$ and $y$" makes sense here.}
\querytext{Q2}{Is "convergent element" defined or explained earlier?}
\querytext{Q3}{Please check if "for the set of bounded
martingales $M$ to be a Banach lattice" could be better that "where the set ...".}
\querytext{Q4}{Please check if "as it is an $\mathcal
\string{E\string}$-martingale" could be better than "as a result an $\mathcal \string{E\string}$-martingale".}
\querytext{Q5}{Please check the edited sentence.}
\end{aqf}
\begin{document}

\begin{frontmatter}
\pretitle{Research Article}

\title{Martingale-like sequences in Banach lattices}

\author[a]{\inits{H.}\fnms{Haile}~\snm{Gessesse}\thanksref{cor1}\ead[label=e1]{hailegessesse@trentu.ca}}
\thankstext[type=corresp,id=cor1]{Corresponding author.}
\author[b]{\inits{A.}\fnms{Alexander}~\snm{Melnikov}\ead[label=e2]{melnikov@ualberta.ca}}

\address[a]{Department of Mathematics,
\institution{Trent University},
Peterborough,~ON,~K9L\,0G2,~\cny{Canada}}
\address[b]{Department of Mathematical and Statistical Sciences,
\institution{University of Alberta}, Edmonton, AB, T6G\,2G1, \cny{Canada}}

\markboth{H. Gessesse, A. Melnikov}{Martingale-like sequences in Banach lattices}




\begin{abstract}
Martingale-like sequences in vector lattice
and Banach lattice frameworks are defined in the same way as martingales are
defined in [Positivity 9 (2005), 437--456]. 
In these frameworks, a collection of bounded $X$-martingales is shown to be a Banach space
under the supremum norm, and under some conditions it is also a Banach
lattice with coordinate-wise order. Moreover, a necessary
and sufficient condition is presented for the collection of
$\mathcal{E}$-martingales to be a vector lattice with coordinate-wise order.
It is also shown that the collection of bounded $\mathcal{E}$-martingales is
a normed lattice but not necessarily a Banach space under the supremum norm.
\end{abstract}
\begin{keywords}
\kwd{Banach lattices}
\kwd{martingales}
\kwd{$E$-martingales}
\kwd{$X$-martingales}
\end{keywords}
\begin{keywords}[MSC2010]%
\kwd[Primary ]{60G48}
\kwd[; Secondary ]{46A40}
\kwd{46B42}
\end{keywords}

\received{\sday{15} \smonth{4} \syear{2018}}
\revised{\sday{7} \smonth{10} \syear{2018}}
\accepted{\sday{8} \smonth{10} \syear{2018}}
\publishedonline{\sday{7} \smonth{11} \syear{2018}}
\end{frontmatter}

\section{Introduction}
The classical definition of martingales is extended to a more general
case in the space of Banach lattices by V.~Troitsky \cite
{TroitskyMartingales:05}. In the Banach lattice framework,
martingales are defined without a probability space and the famous
Doob's convergence theorem was reproduced. Moreover, under certain
conditions on the Banach lattice, it was shown that the set of bounded
martingales forms a Banach lattice with respect to the point-wise order.
In 2011, H.~Gessesse and V.~Troitsky \cite{GessesseMartingales:11}
produced several\vadjust{\goodbreak} sufficient conditions for the space of bounded
martingales on a Banach lattice to be a Banach lattice itself. They
also provided examples showing that the space of bounded martingales is
not necessarily a vector lattice. Several other works have been done by
other authors with regard to martingales in vector lattices, such as
\cite{Watson:13, Grobler:15}.

In the theory of random processes,
not just the study of martingale convergence is important,
but the study of convergence of martingale-like stochastic sequences and processes, and the determination of interrelation between them are also crucial.
So it is natural to ask if martingale-like
sequences can be defined in a vector lattice or Banach lattice
framework. In this article, we define and study martingale-like
sequences in Banach lattices along the same lines as martingales are
defined and studied in \cite{TroitskyMartingales:05}.

Classically, a martingale-like sequence is defined as follows (for
instance, see a paper by A.~Melnikov \cite{Melnikov:82}). Consider a
probability space $(\varOmega,\mathcal{F},P)$ and a filtration
$(\mathcal{F}_n)_{n=1}^{\infty}$, i.e., an increasing sequence of
complete sub-sigma-algebras of $\mathcal{F}$. An integrable stochastic
sequence $x=(x_n,\mathcal{F}_n)$ is an {\bf$L^1$-martingale} if
\[
\lim_{n\rightarrow\infty} \sup_{m\geq n}E\big|E(x_m |
\mathcal {F}_n) - x_n\big| = 0.
\]
An integrable stochastic sequence $x=(x_n,\mathcal{F}_n)$ is an {\bf
$E$-martingale} if
\[
P\bigl\{ \omega: E(x_{n+1} | \mathcal{F}_n) \neq
x_n \text{ infinitely often } \bigr\}=0.
\]

Here we extend the definition of $L^1$-martingales and $E$-martingales
in a general Banach lattice $X$ following the same lines as the
definition of martingales in Banach lattices in \cite
{TroitskyMartingales:05}. First we mention some terminology and
definitions from the theory of Banach lattices for the reader
convenience. For more detailed exploration, we refer the reader to
\cite{Aliprantis:85}. A {\bf vector lattice} is a vector space
equipped with a lattice order relation, which is compatible with the
linear structure. A {\bf Banach lattice} is a vector lattice with a
Banach norm which is monotone, i.e., $0\leq x \leq y$ implies $\norm
{x}\leq\norm{y}$, and satisfies $\norm{x}=\querymark{Q1} \lVert\abs{x}
\rVert$ for any two vectors $x$ and $y$. A vector lattice is said to
be {\bf order complete} if every nonempty subset that is bounded above
has a supremum. We say that a Banach lattice has {\bf order continuous
norm} if $\norm{x_{\alpha}}\rightarrow0$ for every decreasing net
$(x_{\alpha})$ with $\inf x_{\alpha}=0$. A Banach lattice with order
continuous norm is order complete. A sublattice $Y$ of a vector lattice
is called an (order) {\bf ideal} if $y\in Y$ and $|x|\leq|y|$ imply
$x\in Y$. An ideal $Y$ is called a {\bf band} if
$x = \sup_{ \alpha} x_{\alpha}$ implies $x\in Y$ for every positive
increasing net $(x_{\alpha})$ in $Y$. Two elements $x$ and $y$ in a
vector lattice are said to be {\bf disjoint} whenever $|x|\wedge|y| =
0$ holds. If $J$ is a nonempty subset of a vector lattice, then its
{\bf disjoint complement} $J^d$ is the set of all elements of the
lattice, disjoint to every element of $J$. A band $Y$ in a vector
lattice $X$ that satisfies $X = Y\otimes Y^d$ is refered to as a {\bf
projection band}. Every band in an order complete vector lattice is a
projection band. An operator $T$ on a vector lattice X is positive if
$Tx\geq0$ for every $x\geq0$. A sequence of positive projections
$(E_n)$ on a vector lattice $X$ is called a {\bf filtration} if $E_nE_m
= E_{n\wedge m}$. A sequence of positive contractive projections
$(E_n)$ on a normed lattice $X$ is called a {\bf contractive
filtration} if $E_nE_m = E_{n\wedge m}$. A~filtration $(E_n)$ in a
normed lattice $X$ is called \term{dense} if $E_nx\rightarrow x$ for
each $x$ in $X$. In many articles such as in \cite
{TroitskyMartingales:05}, a \term{martingale} with respect to a
filtration $(E_n)$ in a vector lattice $X$ is defined as a sequence
$(x_n)$ in $X$ such that $E_nx_m=x_n$ whenever $m\ge n$.

\section{Main definitions}

\begin{definition}\label{def1}
A sequence $ (x_n)$ of elements of a normed lattice $X$ is called an
{\bf $X$-martingale} relative to a contractive filtration $(E_n)$ if
\[
\lim\limits
_{n\rightarrow\infty} \sup_{m\geq n}\norm{E_nx_m
- x_n} = 0.
\]
\end{definition}

\begin{definition}\label{def2}
A sequence $ (x_n)$ of elements of a vector lattice $X$ is called an
$\mathcal{E}$-\textbf{martingale} relative to a filtration $(E_n)$ if
there exists $n\geq1$ such that $E_m x_{m+1}=x_m$ for all $m\geq n.$
\end{definition}

Note that Definition~\ref{def2} is equivalent to saying a sequence
$(x_n)$ is an $\mathcal{E}$-martingale if there exists $l\geq1$ such
that $E_n x_{m}=x_n$ whenever $m\geq n \ge l$. The symbol ``$\mathcal
{E}$'' stands for eventual so when we say $(x_n)$ is an $\mathcal
{E}$-martingale, we are saying that after a first few finite elements
of the sequence, the sequence becomes a martingale.

Sequences defined by Definition~\ref{def1} and Definition~\ref{def2}
are collectively called {\bf martin\-gale-like sequences}.
Notice that every martingale $(x_n)$ in a vector lattice $X$ with
respect to a filtration $(E_n)$ is obviously an $\mathcal
{E}$-martingale with respect to the filtration
$(E_n)$. Moreover, every $\mathcal{E}$-martingale $(x_n)$ in a Banach
lattice $X$ with respect to a contractive filtration $(E_n)$ is an
$X$-martingale with respect to the contrative filtration
$(E_n)$.
Note that for every $x$ in a vector lattice $X$ and a filtration
$(E_n)$ in $X$, the sequence $(E_n x)$ is an $\mathcal{E}$-martingale
with respect to the filtration $(E_n)$. If $x$ is in a normed space $X$
and $(E_n)$ is a contractive filtration, then the sequence $(E_n x)$ is
an $X$-martingale with respect to the contractive filtration $(E_n)$.

By considering any nonzero martingale $(x_n)$ in a Banach lattice $X$
with respect to filtration $(E_n)$ where $x_1$ is nonzero without loss
of generality, we can define a sequence $(y_n)$ such that $y_1=2x_1$
and $y_n=x_{n}$ for all $n\geq2$. Then one can see that $(y_n)$ is an
$\mathcal{E}$-martingale with respect to the filtration $(E_n)$.
However, $(y_n)$ is not a martingale.

Note that every sequence which converges to zero is an $X$-martingale
with respect to any contractive filtration $(E_n)$ because if
$x_n\rightarrow0$ and $m> n$ then
$\norm{E_n x_m - x_n}\leq\norm{x_m} +\norm{x_n} \rightarrow0$ as
$n\rightarrow\infty$. So one can easily create an $X$-martingale
$(x_n)$ which is not $\mathcal{E}$-martingale by setting $x_n=\frac
{1}{n}x$ where $x$ is a nonzero vector in $X$.

A martingale-like sequence $A=(x_n)$ with respect to a contractive
filtration $(E_n)$ on a normed lattice $X$ is said to be {\bf bounded}
if its norm defined by $\norm{A}=\sup_n\norm{x_n}$ is finite. Given
a contractive filtration $(E_n)$ on a normed lattice $X$, we denote the
set of all bounded $X$-martingales with respect to the contractive
filtration $(E_n)$ by $M_X=M_X(X,(E_n))$ and the set of all bounded
$\mathcal{E}$-martingales with respect to the contractive filtration
$(E_n)$ by $M_E=M_E(X,(E_n))$. With the introduction of the sup norm in
these spaces, one can show that $M_X$ and $M_E$ are normed spaces.
Keeping the notation $M$ of \cite{TroitskyMartingales:05} for all
bounded martingales with respect to the contractive filtration $(E_n)$
and from the preceding arguments, these spaces form a nested increasing
sequence of linear subspaces $M \subset M_E \subset M_X \subset\ell
_\infty(X)$, with the norm being exactly the $\ell_\infty(X)$ norm.\looseness=-1

\begin{theorem}\label{Mx-BS}
Let $(E_n)$ be a contractive filtration on a Banach lattice $X$, then
the collection of $X$-martingales $M_X$ is a closed subspace of $\ell
_\infty(X)$, hence a Banach space.\looseness=-1
\end{theorem}
\begin{proof}
Suppose a sequence $(A^m)=(x^m_n)$ of $X$-martingales converges to $A$
in $\ell_\infty(X)$. We show $A$ is also an $X$-martingale. Indeed,
from $\norm{A^m-A}=\sup_n\norm{x^m_n-x_n}\rightarrow0$ as
$m\rightarrow\infty$, we have that for each $n\geq1$, $\norm
{x^m_n-x_n}\rightarrow0$ as $m\rightarrow\infty$.
Note that for $l\ge n$,
\begin{align*}
\norm{E_nx_l-x_n}&=\norm{E_nx_l-E_nx^m_l+E_nx^m_l-x^m_n+x^m_n-x_n}
\\
&\leq\norm{E_nx_l-E_nx^m_l}
+\norm{E_nx^m_l-x^m_n}
+\norm{x^m_n-x_n}.
\end{align*}
From these inequalities and the contractive property of the filtration,
we have
\[
\lim_{n\rightarrow\infty} \sup_{l\geq n}\norm{E_nx_l
- x_n} = 0.\qedhere
\]
\end{proof}

\begin{corollary}\label{inclusion}
Let $(E_n)$ be a contractive filtration on a Banach lattice $X$, then
$\overline{M_E} \subset M_X.$
\end{corollary}

\begin{lemma}\label{conv-mx}
Let $(E_n)$ be a contractive filtration on a Banach lattice $X$ and
$A=(x_n)$ be in $M_X$ where $x_n\rightarrow x$. Then
\[
\lim\limits
_{n\rightarrow\infty} \sup_{m\ge n}\norm{E_mx -
x_m} = 0.
\]
\end{lemma}

\begin{proof}
Let $A=(x_n)$ be in $M_X$ where $x_n\rightarrow x$. Thus, for $m\geq n$
\[
\norm{E_nx - x_n}= \norm{E_nx -
E_nx_m+E_nx_m-x_n}
\leq\norm{x-x_m} +\norm{E_nx_m-x_n}.
\]
Taking $\lim\limits_{n\rightarrow\infty} \sup_{m\geq n}$ on both
sides of the inequality completes the proof.
\end{proof}

The following proposition confirms that for any convergent element\querymark{Q2}
$A=(x_n)$ of $M_X$ we can find a sequence in $M_E$ that converges to $A$.
\begin{proposition}\label{halfdense}
Let $(E_n)$ be a contractive filtration on a Banach lattice $X$ and
$A=(x_n)$ be a sequence in $M_X$ such that $x_n\rightarrow x$. Then
there exists a sequence $A^m$ in $M_E$ such that $A^m \rightarrow A$ in
$\ell_{\infty}(X)$.
\end{proposition}
\begin{proof}
Suppose $x_n \rightarrow x$ as $n\rightarrow\infty$. First note that
the sequence $(E_nx)$ is in $M$. Now define $A^m=(a^m_n)$ such that
\[
a^m_n= %
\begin{cases}
 x_n,&  \text{for } n\le m,\\
 E_nx,& \text{for } n > m.
\end{cases} %
\]
Then $A^m\in M_E$ and $A^m \rightarrow A$ in $\ell_{\infty}(X)$,
hence $A\in\overline{M_E}$. Indeed, by Lemma~\ref{conv-mx},
\[
\lim_{m\rightarrow\infty}\norm{A^m-A}=\lim_{m\rightarrow\infty
}
\sup_j\norm{E_{m+j}x-x_{m+j}}=0.\qedhere
\]
\end{proof}

In \cite{TroitskyMartingales:05} and \cite{GessesseMartingales:11}
several sufficient conditions are es\querymark{Q3}tablished where the set of bounded
martingales $M$ is a Banach lattice. In \cite{GessesseMartingales:11},
counter examples are provided where $M$ is not a Banach lattice. So,
one may similarly ask when are $M_X$ and $M_E$ Banach spaces and Banach
lattices? We start by showing a counter example that illustrates that
$M_E$ is not necessarily a Banach space.
\begin{example}
Let $c_0$ be the set of sequences converging to zero. Consider the
filtration $(E_n)$ where $E_n \sum_{i=1}^{\infty} \alpha_i e_i =
\sum_{i=1}^{n} \alpha_i e_i$. Thus the sequence $(y_n)$ where $y_n=
\sum_{i=1}^{n} \frac{1}{i} e_i$ is an $E$-martingale with respect to
this filtration. We define a sequence of $E$-martingales $A^m$ as
$A^m=(x_n^m)$ where
\[
x_n^m= %
\begin{cases}
 \sum_{i=n}^{\infty} \frac{1}{i} e_i  ,&  \text{for } n\leq m,\\
 y_n/m,& \text{for } n>m.
\end{cases} %
\]
Define a sequence $A=(x_n)$ where $x_n=\sum_{i=n}^{\infty} \frac
{1}{i} e_i $. We can see that $A$ is not an $E$-martingale. But one can
show that $A^m$ converges to $A$. Indeed,
\[
\bigl\lVert A^m - A \bigr\rVert=\sup_{n}
\big\|x^m_n-x\big\|=\sup_{n \in
\{m+1,m+2, \ldots\}}
\Bigg\|y_n/m-\sum_{i=n}^{\infty}
\frac{1}{i} e_i \Bigg\| \rightarrow0
\]
as $m\rightarrow\infty.$
\end{example}

\section{When is $M_E$ a vector lattice?}

Given a vector (Banach) lattice $X$ and a filtration (respectively
contractive) $(E_n)$ on $X$, we can introduce order structure on the
spaces $M_E$ and $M_X$ as follows. For two bounded $\mathcal
{E}$-martingales (respectively $X$-martingales) $A=(x_n)$ and
$B=(y_n)$, we write $A\geq B$ if $x_n\geq y_n$ for each $n$. With this
order $M_E$ and $M_X$ are ordered vector spaces and the monotonicity of
the norm follows from the monotonicity of the norm of $X$, i.e. for two
$\mathcal{E}$-martingales (respectively $X$-martingales) with $0\leq A
\leq B$, we have $\norm{A}\leq\norm{B}$. For two $\mathcal
{E}$-martingales (respectively $X$-martingales) $A=(x_n)$ and
$B=(y_n)$, one may guess that $A\lor B$ (or $A\wedge B$) can be
computed by the formulas $A\lor B =(x_n\lor y_n)$ (or $A\wedge
B=(x_n\wedge y_n)$). We show in the following theorem that this is in
fact the case in order for $M_E$ to be a vector lattice. However, this
is not obvious to show in the case of $M_X$.

\begin{theorem}\label{vl-equivalence}
Let $X$ be a vector lattice. Then the following statements are equivalent.
\begin{enumerate}
\item[(i)]$M_E$ is a vector lattice.
\item[(ii)] For each $A=(x_n)$ in $M_E$, the sequence $(|x_n|)$ is an
$\mathcal{E}$-martingale and $\abs{A}=(\abs{x_n})$.
\item[(iii)] $M_E$ is a sublattice of $\ell_{\infty}(X)$.
\end{enumerate}
\end{theorem}

\begin{proof}
First we show (\textit{i})~$\implies$~(\textit{ii}). Suppose $M_E$ is a vector lattice and
$A=(x_n)$ is in $M_E$. Since $M_E$ is a vector lattice, $\abs{A}$
exists in $M_E$, say $|A|=B:=(y_n)$. Since $\pm A\le B$, for each $n$,
$\pm x_n\le y_n$. So, $\abs{x_n}\le y_n$ for each $n$. Since $B$ is in
$M_E$, there exists $l$ such that $E_ny_m=y_n$ whenever $m\ge n\ge l$.
Now we claim that $y_n=|x_n|$ for each $n$. Fix $k> l$. We show
$y_n=|x_n|$ for each $n\le k$.

Indeed, define an $\mathcal{E}$-martingale $C=(z_n)$ where
\[
z_n= %
\begin{cases}
 \abs{x_n},& \text{for } n\le k,\\
 y_n, &\text{for } n> k.
\end{cases} %
\]
Since $k>l$ we can easily see that $C$ is an $\mathcal{E}$-martingale.
Moreover, $C\ge0$ and $\pm A \le C \le B$. Since $\abs{A}=B$, $C=B$.
Thus, for every $n\le k$, $y_n=|x_n|$. This establishes (\textit{ii}).

(\textit{ii})~$\implies$~(\textit{iii})~$\implies$~(\textit{i}) is straightforward.
\end{proof}

Using the equivalence in Theorem~\ref{vl-equivalence}, the following
examples illustrate 
that $M_E$ is not always a vector lattice.

\begin{example}\label{ReviewerExample}
Consider the classical martingale $(x_n)$ in $L_1[0,1]$ where
$x_n=\break2^n\mathbf{1}_{[0,2^{-n}]} -\mathbf{1}$ with the filtration
$(\mathcal{F}_n)$ where $\mathcal{F}_n$ is the smallest sigma algebra
generated by the set
\[
\bigl\{ \bigl[0,2^{-n}\bigr], (2^{-n}, 2^{-n+1}],
\dots, (1-2^{-n}, 1] \bigr\}.
\]
One can easily show that
\[
E_n\abs{x_{n+1}}=E \bigl[|x_{n+1}| |x_n \bigr]
\ne|x_n|
\]
for
every $n$ and the sequence $(|x_n|)$ fails to be an $\mathcal
{E}$-martingale. Hence, Theorem~\ref{vl-equivalence} implies that
$M_E$ is not a vector lattice.
\end{example}

\begin{example}\label{cexample}
Consider the filtration $(E_n)$ defined on $c_0$ as follows. For each
$n=0, 1, 2, \ldots$
\[
E_n = %
\begin{bmatrix}
1 & & & & & & & \\
& \ddots & & & & & &\\
& & 1 & & & & &\\
& & &1/2 &1/2 & & &\\
& & &1/2 &1/2 & & &\\
& & & & &1/2 &1/2 &\\
& & & & &1/2 &1/2 &\\
& & & & & & &\ddots\\
\end{bmatrix} %
\]
with $2n$ ones in the upper left corner. For each $e_i=(0,\ldots, 0,
\underbrace{1}_{i^{\text{th}}}, 0, \ldots)$,
$E_ne_i=e_i$ if $i\leq2n$ and $E_ne_{2k-1}=E_ne_{2k}=\frac
{1}{2}(e_{2k-1}+e_{2k})$ if $n<k$. Now if we define a sequence
$A=(x_n)$ where for each $n=0,1,2,\ldots$,
\[
x_n=(\underbrace{-1, 1, \dots, -1, 1}_{2n\text{-tuple}}, 0, \dots),
\]
one can show thi\querymark{Q4}s is a martingale as a result an $\mathcal
{E}$-martingale. However, $|A|=(|x_n|)$ where
\[
|x_n|=(\underbrace{1, \dots, 1}_{2n\text{-tuple}}, 0,\dots)
\]
is not an $\mathcal{E}$-martingale. So, Theorem~\ref{vl-equivalence}
implies that $M_E$ is not a vector lattice.
\end{example}

\begin{proposition}
If a filtration $(E_n)$ is a sequence of band projections, then $M_E$
is a vector lattice with coordinate-wise lattice operations.
\end{proposition}

\begin{proof}
If $A=(x_n)\in M_E$, then there exists $l$ such that $E_nx_m=x_n$
whenever $m\ge n \ge l$. Thus, $E_n\abs{x_m}=\abs{E_nx_m}=\abs
{x_n}$. So, $|A|=(\abs{x_n})$ and thus $M_E$ is a vector lattice.
\end{proof}

\begin{theorem}\label{filt-char}
If $M_E$ is a normed lattice and the filtration $(E_n)$ is dense in
$X$, then for each $x$ in $X$, there exists $l$ such that $\abs
{E_nx}=E_n\abs{x}$ whenever $n\ge l$.
\end{theorem}

\begin{proof}
Let $x$ be in $X$. Then $(E_n)$ is dense means $E_n x\rightarrow x$.
Moreover, $(E_nx)$ is a martingale. Since $M_E$ is a vector lattice, by
Theorem~\ref{vl-equivalence}, $(\abs{E_n x})$
is an $E$-martingale. Thus there exists $l$ such that for any $m$ and
$n$ with $m\ge n\ge l$, $\abs{E_n E_m x}=\abs{E_n x}$ and
$E_n\abs{ E_m x}=\abs{E_n x}$. So, $\abs{E_n E_m x}=E_n\abs{ E_m
x}$ and letting $m\rightarrow\infty$, we have $\abs{E_n x}=E_n\abs{ x}$.
\end{proof}

\section{When is $M_X$ a Banach lattice?}

Under the pointwise order structure on $M_X$, for an $X$-martingale
$A=(x_n)$, we can refer to Example~\ref{cexample} to show that the sequence
$(|x_n|)$ is not necessarily an $X$-martingale. However, under certain
assumptions, we can show that $(\abs{x_n})$ is \xch{an}{an an} $X$-martingale
for every $X$-martingale $A=(x_n)$ making $M_X$ a Banach lattice.

\begin{proposition}
If $(E_n)$ is a contractive filtration where $E_n$ is a band projection
for every $n$ then $M_X$ is a Banach lattice with coordinate-wise
lattice operations.
\end{proposition}

\begin{proof}
Let $A=(x_n)$ be an $X$-martingale. For each $n$ and $m$, $E_n$ is a
band projection implies $E_n|x_m|=|E_nx_m|$.
Thus, by the fact that $\Bigabs{ |x|-|y|}\leq\abs{x-y}$,
for $m\geq n$,
\[
\big\|E_n|x_m|-|x_n|\big\|=
\big\||E_nx_m|-|x_n|\big\|\leq
\norm{E_nx_m-x_n}.
\]
This implies
\[
\lim_{n\rightarrow\infty} \sup_{m\geq n}\big\|E_n|x_m|
- |x_n|\big\| = 0
\]
which implies $|A|=(|x_n|)$ is also an $X$-martingale.
\end{proof}

\begin{question}
From Theorem~\ref{vl-equivalence}, $M_E$ is a vector lattice if and
only if for each $\mathcal{E}$-martingale $(x_n)$, the sequence $(\abs
{x_n})$ is also an $\mathcal{E}$-martingale. This is the case when the
filtration is a sequence of band projections. Can one give a
characterization of the filtrations for which $M_E$ is a vector
lattice? Or, can one give an example of a filtration which
is not a sequence of projections
and the corresponding $M_E$ is a vector lattice?
\end{question}




\end{document}